\documentclass[reqno]{amsart}

\usepackage{amsmath}
\usepackage{amssymb}
\usepackage{amsthm}
\usepackage{xspace}
\usepackage{pictexwd}
\usepackage{enumerate}
\usepackage{graphicx}
\usepackage{cancel}
\usepackage{pinlabel}

\theoremstyle{plain}

\newtheorem{prop}{Proposition}
\newtheorem{lem}{Lemma}

\theoremstyle{definition}

\theoremstyle{remark}
\newtheorem*{rem}{Remark}
\newtheorem*{example}{Example}

\def\forcehmode{\hskip0pt\relax}

\let\myskip=\medskip

\def\definebb#1=#2.{\def#1{{{\mathbb #2}^{\vphantom{x}}}}}
\definebb \c=C.
\definebb \q=Q.
\definebb \r=R.
\definebb \s=S.
\definebb \z=Z.

\def\chp{H_{\c}^2}

\def\calN{\mathcal{N}}

\def\ppp{$(p_1,p_2,p_3)$}

\def\al{{\alpha}}

\def\Ga{{\Gamma}}
\def\ga{{\gamma}}
\def\De{{\Delta}}

\def\La{{\Lambda}}

\def\Gaprimstar{\Ga'\backslash\{\id\}}
\def\Lastar{\La\backslash\{\id\}}
\def\fa{{f_A}}
\def\fb{{f_B}}
\def\wa{w_A}
\def\wal{{\wa^{(\ell)}}}
\def\wak{{\wa^{(k)}}}
\def\wb{{w_B}}

\let\emptyset=\varnothing
\def\dd{\partial}
\def\st{\,\,\big|\,\,}
\def\<{\langle}
\def\>{\rangle}

\let\ge=\geqslant
\let\le=\leqslant

\def\defit{\it}

\DeclareMathOperator{\dist}{dist}
\DeclareMathOperator{\id}{id}
\DeclareMathOperator{\PU}{PU}
\DeclareMathOperator{\SU}{SU}
\DeclareMathOperator{\trace}{trace}

\let\Im=\undefined \DeclareMathOperator{\Im}{Im}
\let\Re=\undefined \DeclareMathOperator{\Re}{Re}

\author[Andrew Monaghan]{Andrew Monaghan}
\address{Department of Mathematical Sciences, University of Liverpool, Peach Street, Liverpool L69~7ZL, UK}
\email{a.monaghan05@googlemail.com}
\author[John Parker]{John R.~Parker}
\address{Department of Mathematical Sciences, Durham University, Science Laboratories, South Road, Durham DH1~3LE, UK}
\email{j.r.parker@durham.ac.uk}
\author[Anna Pratoussevitch]{Anna Pratoussevitch}
\address{Department of Mathematical Sciences, University of Liverpool, Peach Street, Liverpool L69~7ZL, UK}
\email{annap@liverpool.ac.uk}

\title[Ultra-Parallel Complex Hyperbolic Triangle Groups]
      {Discreteness of Ultra-Parallel Complex Hyperbolic Triangle Groups of Type $[m_1,m_2,0]$}

\begin{date}  {\today} \end{date}

\thanks{A.M. acknowledges the financial support from an EPSRC DTA scholarship at the University of Liverpool.
J.R.P. acknowledges support from U.S. National Science Foundation
grants DMS 1107452, 1107263, 1107367 "RNMS: Geometric structures And
Representation varieties" (the GEAR Network).
A.P. acknowledges support from U.S. National Science Foundation conference grant for "Geometries, Surfaces and Representation of Fundamental Groups" at the University of Maryland.}



\begin{document}

\maketitle

\begin{abstract}
In this paper we consider ultra-parallel complex hyperbolic triangle groups of type $[m_1,m_2,0]$,
i.e.\ groups of isometries of the complex hyperbolic plane,
generated by complex reflections in three ultra-parallel complex geodesics
two of which intersect on the boundary.
We prove some discreteness and non-discreteness results for these groups
and discuss the connection between the discreteness results and ellipticity of certain group elements.
\end{abstract}

\section{Introduction}

\myskip
Complex hyperbolic triangle groups are groups of isometries of the complex hyperbolic plane~$\chp$,
generated by complex reflections in three complex geodesics.
For groups of complex hyperbolic isometries, the main obstacle to discreteness is the presence of elliptic elements of infinite order.
More precisely, a group of holomorphic isometries of~$\chp$ without stable proper totally geodesic subspaces and without elliptic elements of infinite order is discrete (see~\cite{CG, Gol, Will}).

\myskip
For a triple $p_1, p_2, p_3$, where each of the numbers~$p_k$ can be either a positive integer or equal to~$\infty$,
we say that a {\defit complex hyperbolic \ppp-triangle group representation\/} is a representation of the group
$$\<\ga_1,\ga_2,\ga_3\st\ga_k^2=(\ga_{k-1}\ga_{k+1})^{p_k}=1,~k=1,2,3\>$$
(where $\ga_{k+3}=\ga_k$, and the relation $(\ga_{k-1}\ga_{k+1})^{p_k}=1$ is to be omitted when $p_k=\infty$)
into the group $\PU(2,1)$ of holomorphic isometries of~$\chp$,
given by taking the generators $\ga_1,\ga_2,\ga_3$ to complex reflections $I_1,I_2,I_3$ of order~$2$ in complex geodesics $C_1,C_2,C_3$ in~$\chp$ such that $C_{k-1}$ and $C_{k+1}$ meet at the angle $\pi/p_k$ when $p_k$ is finite
resp.\ at the angle~$0$ when $p_k$ is equal to~$\infty$.
R.~Schwartz in his ICM talk in~2002~\cite{Sch02} conjectured that a complex hyperbolic \ppp-triangle group representation is discrete and faithful if and only if a group element~$w$ is non-elliptic,
where $w=\wa=I_1I_2I_1I_3$ or $w=\wb= I_1I_2I_3$ depending on~\ppp.

\myskip
In this paper we will consider instead the case of groups generated by complex reflections in complex geodesics that do not intersect inside~$\chp$.
In this case we will show that it is necessary to consider a larger set of elements, in fact infinitely many, 
$\wal=I_1(I_2I_1)^{\ell}I_3$ for~$\ell\in\z$ and~$\wb=I_1I_2I_3$,
and we will prove a generalisation of Schwartz' conjecture in a special case.

\myskip
For a triple $m_1, m_2, m_3$ of non-negative real numbers,
we say that a {\defit complex hyperbolic ultra-parallel $[m_1,m_2,m_3]$-triangle group\/} is a subgroup of $\PU(2,1)$ generated by complex reflections $I_1,I_2,I_3$ of order~$2$
in complex geodesics $C_1,C_2,C_3$ in~$\chp$ such that the distance between the closures of $C_{k-1}$ and $C_{k+1}$ in~$\chp$ is equal to~$m_k$.
A {\defit complex hyperbolic $[m_1,m_2,m_3]$-triangle group representation\/} is a representation of the group $\Ga=\<\ga_1,\ga_2,\ga_3\st\ga_k^2=1,~k=1,2,3\>=(\z/2\z)^{*3}$ into the group $\PU(2,1)$ given by taking the generators $\ga_k$ to the generators $I_k$ of an $[m_1,m_2,m_3]$-triangle group.
The deformation space of $[m_1,m_2,m_3]$-triangle groups for given distances~$m_1,m_2,m_3$ is of real dimension one,
such a group is determined up to an isometry by the angular invariant~$\al\in[0,2\pi]$,
see section~\ref{def-angular-invariant} for a definition.
Some special cases of ultra-parallel triangle groups have been considered previously, such as $[m,m,0]$-groups and $[m,m,2m]$-groups in~\cite{WG} and $[m,m,m]$-groups in~\cite{Vas}.

\myskip
Our results on complex hyperbolic $[m_1,m_2,0]$-triangle group representations are summarized in Figure~\ref{fig-regions1}.

\begin{figure}[htbp]
\labellist
\pinlabel {$\frac{2}{3}$} [t] at 178 147
\pinlabel {$\frac{2}{4}$} [t] at 145 147
\pinlabel {$\frac{2}{5}$} [t] at 130 147
\pinlabel {$\frac{2}{3}$} [t] at 178 147
\pinlabel {$(2,\frac{1}{2})$} [br] at 435 500
\pinlabel {$(1,\frac{1}{6})$} [br] at 237 265
\pinlabel {$(\frac{2}{3},\frac{1}{12})$} [br] at 185 215
\endlabellist
\begin{center}
\includegraphics[width=0.5\textwidth]{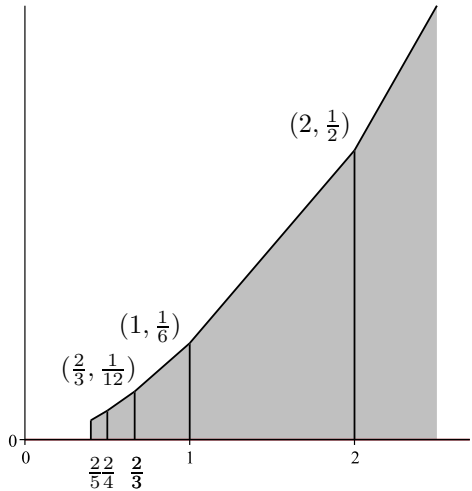}
\caption{Conditions of Propositions~\ref{conditions2-typeA} and~\ref{conditions2-typeB}}
\label{fig-regions1}
\end{center}
\end{figure}

\noindent
The upper right quadrant represents all pairs of distances $m_1$ and $m_2$,
where $(m_1,m_2)$ with $m_1\ge m_2>0$ corresponds to the point with the coordinates
$$(X,Y)=\left(\frac{\cosh^2(m_1/2)-1}{\cosh^2(m_2/2)-1}-1,\frac{1}{\cosh^2(m_2/2)-1}\right).$$
Consider the piecewise linear curve that consists of segments between the points $\left(\frac{2}{k},\frac{1}{k(k+1)}\right)$, $k\in\z$, $k\ge1$ and the ray starting at the point $\left(2,\frac{1}{2}\right)$ with the gradient~$\frac{1}{2}$.
Only the points and segments up to~$k=5$ are shown in the figure, but the broken line continues to the left.
Proposition~\ref{conditions2-typeA} states that the shaded region below the broken line
corresponds to pairs~$(m_1,m_2)$ such that $[m_1,m_2,0]$-representations are discrete and faithful 
if and only if the element~$\wa^{(k)}$ is non-elliptic,
where $k=1$ corresponds to the part of the shaded region with~$X\ge2$,
while each $k\ge2$ corresponds to the part  of the shaded region with $\frac{2}{k}\le X\le\frac{2}{k-1}$.
For the unshaded region above the broken line we expect that the elipticity of the element~$\wb$ plays a key role.
Proposition~\ref{conditions2-typeB} says that this region corresponds to pairs~$(m_1,m_2)$
such that $[m_1,m_2,0]$-representations are discrete and faithful 
if $\Re(\trace(\wb))\le-5$ which is a condition sufficient but not necessary for the element~$\wb$ to be non-elliptic.

\begin{prop}
\label{conditions2-typeA}
Suppose that $m_1\ge m_2>0$ and for some~$k\in\z$, $k\ge1$
$$\max\left\{\frac{1}{k}+\frac{k+1}{r_2^2-1},\  \frac{2}{k}\right\}\le\frac{r_1^2-1}{r_2^2-1}-1\le\frac{2}{k-1},$$
where $r_j=\cosh(m_j/2)$, $j=1,2$ and the second inequality is omitted for~$k=1$.
Then a complex hyperbolic $[m_1,m_2,0]$-triangle group representation is discrete and faithful if and only if the element~$\wa^{(k)}=I_1(I_2I_1)^{k}I_3$ is non-elliptic.
\end{prop}


\begin{prop}
\label{conditions2-typeB}
Suppose that either $m_1\ge m_2>0$ and for all~$k\in\z$, $k\ge1$,
$$\frac{r_1^2-1}{r_2^2-1}-1\le\frac{1}{k}+\frac{k+1}{r_2^2-1}$$
or $m_1\ge m_2=0$ and $r_1\le\sqrt{3}$,
where $r_j=\cosh(m_j/2)$, $j=1,2$.
Then a complex hyperbolic $[m_1,m_2,0]$-triangle group representation is discrete and faithful if $\Re(\trace(\wb))\le-5$.
\end{prop}


\myskip
We also demonstrate in Proposition~\ref{discrete} that in some cases Proposition~\ref{conditions2-typeA} can be used to prove discreteness for all values of the angular invariant.
These results are summarized in Figure~\ref{fig-regions3}.
The coordinates and the light shading are as in Figure~\ref{fig-regions1}.
The regions in Proposition~\ref{discrete} correspond to the darkly shaded regions under hyperbolae in Figure~\ref{fig-regions3}.

\begin{figure}[htbp]
\labellist
\pinlabel {$\frac{2}{3}$} [t] at 178 147
\pinlabel {$\frac{2}{4}$} [t] at 145 147
\pinlabel {$\frac{2}{5}$} [t] at 130 147
\pinlabel {$\frac{2}{3}$} [t] at 178 147
\pinlabel {$(2,\frac{1}{2})$} [br] at 435 500
\pinlabel {$(1,\frac{1}{6})$} [br] at 237 265
\pinlabel {$(\frac{2}{3},\frac{1}{12})$} [br] at 185 215
\endlabellist
\begin{center}
\includegraphics[width=0.5\textwidth]{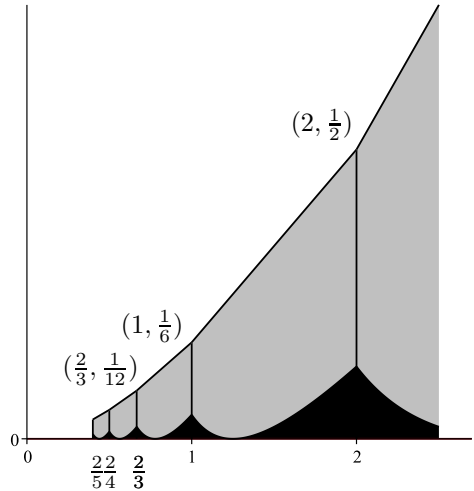}
\caption{Conditions of Proposition~\ref{discrete}}
\label{fig-regions3}
\end{center}
\end{figure}

\begin{prop}
\label{discrete}
Any complex hyperbolic ultra-parallel $[m_1,m_2,0]$-triangle group with $m_1\ge m_2\ge0$ is discrete
if the following condition on $r_j=\cosh(m_j/2)$, $j=1,2$ is satisfied:
$$
r_1-r_2\in\left(\bigcup\limits_{k=2}^{\infty}\left[\frac{r_2+1}{k},\frac{r_2-1}{k-1}\right]\right)
\cup\left[r_2+1,\infty\right).
$$
\end{prop}


We also show the following discreteness test that is easier to check but less powerful than Propositions~\ref{conditions2-typeA} and~\ref{conditions2-typeB}.
It was first proved in the PhD thesis of the first author (\cite{Mo}, Theorem~3.3.0.8, p.~113) and generalises the results by J.~Wyss-Gallifent (\cite{WG}, Chapter~4) about ultra-parallel $[m,m,0]$-triangle groups.

\begin{prop}
\label{conditions1}
A complex hyperbolic ultra-parallel $[m_1,m_2,0]$-triangle group with angular invariant~$\al$ is discrete if 
$$\sin\left(\frac{\al}{2}\right)\ge\frac{1}{r_1+r_2},$$
where $r_j=\cosh(m_j/2)$, $j=1,2$.
\end{prop}


\myskip
In contrast to the discreteness results we also prove the following non-discreteness results obtained using Shimizu's lemma~\cite{Par97}:

\begin{prop}
\label{non-discrete}
A complex hyperbolic ultra-parallel $[m_1,m_2,0]$-triangle group with angular invariant~$\al$ is non-discrete if one of the following conditions on $r_j=\cosh(m_j/2)$, $j=1,2$ is satisfied:
\begin{enumerate}[$\bullet$]
\item
$34r_1^2r_2^2-15r_1^4-15r_2^4+2r_1^2+2r_2^2\ge0$ and
$$
  64r_1r_2\sin^2\left(\frac{\al}{2}\right)
  <32r_1r_2-15r_1^2-15r_2^2+1-\sqrt{34r_1^2r_2^2-15r_1^4-15r_2^4+2r_1^2+2r_2^2}.
$$
\item
$34r_1^2r_2^2-15r_1^4-15r_2^4+2r_1^2+2r_2^2<0$ and
$$64r_1r_2\sin^2\left(\frac{\al}{2}\right)<1-16(r_1-r_2)^2.$$
\end{enumerate}
\end{prop}

The paper is organised as follows:
In section~\ref{sec-basics} we summarise the necessary basics in complex hyperbolic and Heisenberg geometry.
We introduce the standard parametrisation for ultra-parallel $[m_1,m_2,0]$-triangle groups in section~\ref{sec-param}.
In section~\ref{sec-discr-results} we use the compression property to derive discreteness conditions for $[m_1,m_2,0]$-groups
and prove Propositions~\ref{conditions2-typeA}, \ref{conditions2-typeB} and~\ref{conditions1}.
We use these discreteness conditions in section~\ref{sec-discr-for-all-alpha} to show in some cases the discreteness for all values of the angular invariant and prove Proposition~\ref{discrete}.
In section~\ref{sec-non-discr} we use a version of Shimizu's lemma to show some non-discreteness results and prove Proposition~\ref{non-discrete}.
In section~\ref{sec-ellipticity} we recall the conjecture of R.~Schwartz in more detail and put our results in the context of this general conjecture. 
In section~\ref{sec-isosceles} we summarise all our results in the case of isosceles triangles $m_1=m_2$.

\section{Basics}

\label{sec-basics}

\myskip
We will first recall some basic notions of the complex hyperbolic geometry.
For general references on complex hyperbolic geometry and complex hyperbolic triangle groups see~\cite{Gol,Par09,Par10}.

\subsection{Complex hyperbolic plane:}
Let $\c^{2,1}$ denote the vector space $\c^3$ equipped with a Hermitian form of signature~$(2,1)$,
for example $\<z,w\>=z_1\bar w_3+z_2\bar w_2+z_3\bar w_1$.
We call a vector $z\in\c^{2,1}$ {\defit negative}, {\defit null\/} or {\defit positive}
if $\<z,z\>$ is negative, zero or positive respectively.
Let $P(\c^{2,1})$ denote the projectivisation of~$\c^{2,1}-\{0\}$.
We denote the image of $z=(z_1,z_2,z_3)\in\c^{2,1}$ under the projectivisation map by $[z]=[z_1:z_2:z_3]$.
The {\defit complex hyperbolic plane\/} $\chp$ is the projectivisation of the set of negative vectors in~$\c^{2,1}$,
equipped with the {\defit Bergman metric\/} given by
$$\cosh^2\left(\frac{\dist([z],[w])}{2}\right)=\frac{\<z,w\>\< w,z\>}{\<z,z\>\<w,w\>}.$$
The ideal boundary $\dd\chp$ of $\chp$ is defined as the projectivisation of the set of null vectors in~$\c^{2,1}-\{0\}$.

\subsection{Isometries:}
The holomorphic isometry group of~$\chp$ is the projectivisation $\PU(2,1)$
of the group of those complex linear transformations which preserve the Hermitian form.
Isometries can be classified according to their fixed point behaviour,
an isometry is {\defit elliptic\/} if it has at least one fixed point in~$\chp$,
{\defit parabolic\/} if it has one fixed point in~$\dd\chp$
and {\defit loxodromic\/} if it has two fixed points in~$\dd\chp$.
An isometry is called {\defit regular elliptic} if for the corresponding element in~$\SU(2,1)$ all eigenvalues are distinct.
The type of an isometry can be determined from the position of the trace of the corresponding matrix in the complex plane.
The deltoid curve
$$\De=\{z\in\c\st|z|^4-8\Re(z^3)+18|z|^2=27\}$$
has the property that an isometry~$A$ in~$\SU(2,1)$
is regular elliptic if and only if $\trace(A)$ is inside~$\De$ and is loxodromic if and only if $\trace(A)$ is outside~$\De$
(see \cite{Gol}, Theorem~6.2.4).

\subsection{Complex geodesics:}
A {\defit complex geodesic\/} is a projectivisation of a $2$-di\-men\-sio\-nal complex subspace of~$\c^{2,1}$.
Any positive vector $c\in\c^{2,1}$ determines a complex geodesic
$$P(\{z\in\c^{2,1}\st\<c,z\>=0\}).$$
Conversely, any complex geodesic is of this form for some positive vector~$c\in\c^{2,1}$,
called a {\defit polar vector\/} of the complex geodesic.
A polar vector is unique up to multiplication by a complex scalar.
We say that the polar vector~$c$ is {\defit normalised\/} if $\<c,c\>=1$.

\myskip
A typical example is the complex geodesic $\{[z:0:1]\in\chp\}$ with polar vector $c=(0,1,0)$.
Any complex geodesic is isometric to this one.

\myskip
Let $C_1$ and $C_2$ be complex geodesics with normalised polar vectors $c_1$ and $c_2$ respectively.
Then $C_1$ and $C_2$ intersect in $\dd\chp$ if and only if $|\<c_1,c_2\>|=1$.
We call $C_1$ and $C_2$ {\defit ultra-parallel\/} if they have no points of intersection in $\chp\cup\dd\chp$,
in which case $|\<c_1,c_2\>|>1$ and $|\<c_1,c_2\>|=\cosh\left(\frac12\cdot\dist(C_1,C_2)\right)$,
where $\dist(C_1,C_2)$ is the distance between $C_1$ and $C_2$.

\subsection{Complex reflections:}
Given a complex geodesic $C$,
there is a unique isometry $I_C$ in~$\PU(2,1)$ of order~$2$, whose fixed point set is equal to~$C$.
This isometry is called the {\defit complex reflection of order}~$2$ in~$C$ (or {\defit inversion\/} on~$C$) and is given by
$$I_C(z)=-z+2\frac{\<z,c\>}{\<c,c\>}c,$$
where $c$ is a polar vector of~$C$.
(Unlike real reflections, complex reflections in complex geodesics can be of arbitrary order.
We will only treat the order~$2$ case in this paper.)

\subsection{Complex hyperbolic triangle groups:}
A {\defit complex hyperbolic triangle\/} is a triple $(C_1,C_2,C_3)$ of complex geodesics in~$\chp$.
For a triple~$(m_1,m_2,m_3)$, where each of the numbers~$m_j$ is non-negative,
we say that a triangle $(C_1,C_2,C_3)$ is a {\defit complex hyperbolic ultra-parallel $[m_1,m_2,m_3]$-triangle\/}
if the complex geodesics $C_{j-1}$ and $C_{j+1}$ are ultra-parallel at distance~$m_j$.
A {\defit complex hyperbolic ultra-parallel $[m_1,m_2,m_3]$-triangle group\/}
is a subgroup of~$\PU(2,1)$ generated by complex reflections~$I_j$ of order~$2$ in the sides~$C_j$
of a complex hyperbolic ultra-parallel $[m_1,m_2,m_3]$-triangle $(C_1,C_2,C_3)$.

\subsection{The space of complex hyperbolic triangle groups:}
\label{def-angular-invariant}
For a given triple $m_1,m_2,m_3$ the space of $[m_1,m_2,m_3]$-triangles is of real dimension one.
We now describe a parameterisation of the space of complex hyperbolic triangles in~$\chp$
by means of an angular invariant $\al$ (see section~3 in~\cite{Pra} for details).
Let~$(C_1,C_2,C_3)$ be a complex hyperbolic triangle.
Let $c_k$ be the normalised polar vector of the complex geodesic~$C_k$.
We define the {\defit angular invariant}~$\al$ of the triangle $(C_1,C_2,C_3)$ as
$$\al=\arg\left(\prod_{k=1}^3\<c_{k-1},c_{k+1}\>\right).$$
An ultra-parallel complex hyperbolic triangle in $\chp$ is determined uniquely up to isometry
by the three distances between the complex geodesics and the angular invariant~$\al$.
For any~$\al\in[0,2\pi]$ an $[m_1,m_2,m_3]$-triangle with the angular invariant~$\al$ exists if and only if
$$\cos\al<\frac{r_1^2+r_2^2+r_3^2-1}{2r_1r_2r_3},$$
where~$r_j=\cosh(m_j/2)$.
In the case $m_3=0$ we have $r_3=1$ and therefore
$$\frac{r_1^2+r_2^2+r_3^2-1}{2r_1r_2r_3}=\frac{r_1^2+r_2^2}{2r_1r_2}\ge1,$$
thus for every~$\al\in(0,2\pi)$
there exists an $[m_1,m_2,0]$-triangle with the angular invariant~$\al$.


\subsection{Heisenberg group:}
In the same way that the boundary of the real hyperbolic space is the one point compactification of the Euclidean space of one dimension lower,
we may identify the boundary $\dd\chp$ with $\calN=\c\times\r\cup\{\infty\}$,
a one point compactification of the Heisenberg group.
One such homeomorphism taking $\dd\chp$ to $\calN$ is given by the stereographic projection:
$$
[z_1:z_2:z_3]\longmapsto \left(\frac{z_2}{z_3\sqrt{2}},\Im\left(\frac{z_1}{z_3}\right)\right)
\quad\text{if }~z_3\ne0; \qquad [z:0:0]\longmapsto\infty.
$$

\subsection{Chains:}
A complex geodesic is homeomorphic to a disc,
its intersection with the boundary of the complex hyperbolic plane is homeomorphic to a circle.
Circles that arise as the boundaries of complex geodesics are called {\defit chains\/}.
From two distinct points on a chain we can retrieve the complex geodesic through them,
so there is a bijection between chains and complex geo\-de\-sics.
We now discuss the representations of the chains in Heisenberg space~$\calN$, see~\cite{Gol}, \cite{Par10} for more details.
Chains passing through~$\infty$ are represented as vertical straight lines defined by 
$\zeta = \zeta_0$,
such chains are called {\defit vertical\/}.
The vertical chain defined by $\zeta=\zeta_0$ consists of all points 
$[z_1:\sqrt{2}\zeta_0z_2:z_2]$ in $P(\c^{2,1})$. It has normalised polar vector
$[-\sqrt{2}\bar\zeta_0:1:0]$.
A chain not containing~$\infty$ is called {\defit finite\/}.
A finite chain is represented by an ellipse whose vertical projection $\c\times\r\rightarrow\c$
is a (Euclidean) circle in~$\c$.
The finite chain with centre $(\zeta_0,\upsilon_0)\in\calN$ and radius $r_0>0$ has polar vector
$$\big[r_0^2-|\zeta_0|^2+iv_0:\sqrt{2}\zeta_0:1\bigr]$$
and consists of all $(\zeta,\upsilon)\in\calN$ satisfying the equations
$$
  |\zeta-\zeta_0|=r_0,
  \qquad
  \upsilon=\upsilon_0-2\Im(\zeta\bar{\zeta}_0).
$$
In particular, the finite chain with centre~$(0,0)$ and radius~$1$ is the unit circle in the $\c\times\{0\}$ plane
and has polar vector $[1:0:1]$, hence it has a normalised polar vector $[1/\sqrt{2}:0:1/\sqrt{2}]$.

\subsection{Heisenberg isometries:}
The Heisenberg group~$\calN$ is equipped with the {\defit Cygan metric}
$$
\rho_0((\zeta_1,v_1),(\zeta_2,v_2))
=\Big||\zeta_1-\zeta_2|^2-i(v_1-v_2)-2i\Im(\zeta_1\bar\zeta_2)\Big|^{1/2}.
$$
A {\defit Heisenberg translation\/} by~$(\tau,t)\in\calN$ is given by
$$
(\zeta,v)\mapsto(\tau,t)+(\zeta,v)=(\zeta+\tau,v+t+2\Im(\zeta\bar{\tau}))
$$
and corresponds to the following matrix in~$\PU(2,1)$ (see~\cite{Gol}, section~4.2):
$$
\left(\begin{matrix} 1 & -\sqrt{2}\bar\tau & -|\tau|^2+it \\
0 & 1 & \sqrt{2}\tau \\ 0 & 0 & 1 \end{matrix}\right).
$$
There is a bijection between chains and complex geo\-de\-sics.
We can therefore, without loss of generality, talk about reflections in chains instead of reflections in complex geodesics.
An inversion~$I_{C_{\zeta_0}}$ in a vertical chain $C_{\zeta_0}$ which intersects 
$\c\times\{0\}$ at $\zeta_0$ and has the polar vector 
$c_{\zeta_0}=\bigl[-\sqrt{2}\bar\zeta_0:1:0\bigr]$ 
corresponds to the following element in~$\PU(2,1)$:
$$
\left(\begin{matrix}
-1 & -2\sqrt{2}\bar\zeta_0 & 4|\zeta_0|^2 \\ 
0 & 1 & -2\sqrt{2}\zeta_0 \\ 0 & 0 & -1 \end{matrix}\right)
$$
For an element~$h=(h_{ij})_{1\le i,j\le3}\in\SU(2,1)$ with $h(\infty)\ne\infty$ we can define the {\defit isometric sphere} of~$h$ as the sphere with respect to the Cygan metric
with centre~$h^{-1}(\infty)$ and radius $r_h=1/\sqrt{|h_{31}|}$ 
see~\cite{Par97} and section~5.4 in~\cite{Gol}.

\subsection{Products of reflections in chains:}
What effect does an inversion in a vertical chain have on another vertical chain?
Suppose we have vertical chains $C_{\zeta}$ and $C_{\xi}$ which intersect $\c\times\{0\}$ at $\zeta$ and $\xi$ and have polar vectors
$$c_{\zeta}=\bigl[-\sqrt{2}\bar\zeta:1:0\bigr]\qquad\text{and}\qquad c_{\xi} =\bigl[-\sqrt{2}\bar{\xi}:1:0\bigr]$$
respectively.
What effect does the inversion in $C_{\zeta}$ have on $C_{\xi}$?
We calculate
$$
  I_{C_{\zeta}}(z)
  =\left(\begin{matrix}
-1 & -2\sqrt{2}\bar\zeta_0 & 4|\zeta_0|^2 \\ 
0 & 1 & -2\sqrt{2}\zeta_0 \\ 0 & 0 & -1 \end{matrix}\right)
\left(\begin{matrix}-\sqrt{2}\bar{\xi} \\ 1 \\ 0 \end{matrix}\right)
=\left(\begin{matrix}-\sqrt{2}(2\bar\zeta-\bar{\xi}) \\ 1 \\ 0 \end{matrix}\right)
$$
which is a polar vector of the vertical chain that intersects $\c\times\{0\}$ at $2\zeta-\xi$.
Therefore inversion in the vertical chain~$C_{\zeta}$
rotates the vertical chain~$C_{\xi}$ as a set around~$C_{\zeta}$ through~$\pi$.

\subsection{Bisectors and spinal spheres:}
Unlike in the real hyperbolic space, there are no totally geodesic real hypersurfaces in $\chp$.
An acceptable substitute is the collection of metric bisectors:
Let $z_1,z_2\in\chp$ be two distinct points.
The {\defit bisector\/} equidistant from~$z_1$ and~$z_2$ is defined as 
$$\{z\in\chp\st\rho(z_1,z)=\rho(z_2,z)\}.$$


\myskip
A {\defit spinal sphere\/} is an intersection of a bisector with the boundary of~$\chp$.
It is a smooth hypersurface in $\dd\chp$, diffeomorphic to a sphere.

\myskip
An example is the bisector
$$
\mathfrak{C}=\bigl\{[z_1:z_2:z_3]\in\chp\st |z_1|=|z_3|\bigr\}
$$
and its boundary, the {\defit unit spinal sphere}, which can be described as 
$$U=\{(\zeta,\upsilon)\in\calN:|\zeta|^4+\upsilon^2=1\}.$$

\myskip
For more details on bisectors and spinal spheres see~\cite{Gol}.





\subsection{A discreteness criterion:}
Let $I_1$, $I_2$ and $I_3$ be reflections in the complex geodesics $C_1$, $C_2$ and $C_3$ respectively.
Let $\Ga$ be the group generated by $I_1$, $I_2$ and $I_3$.
Let $\Ga'$ be the group generated by $I_1$ and $I_2$.
We say that the group $\Ga$ is {\defit compressing\/} if 
there exist subsets $U_1$, $U_2$, $V$ of $\calN$ with $U_1\cap U_2=\emptyset$ and $V\varsubsetneq U_1$
such that
\begin{enumerate}
\item
$I_3(U_1)=U_2$;
\item
$g(U_2)\subsetneq V$ for all elements~$g\in\Gaprimstar$.
\end{enumerate}

\myskip
We will use the following discreteness criterion used by Schwartz and Wyss-Gallifent~\cite{WG}: If $\Ga$ is compressing, then $\Ga$ is a discrete subgroup of $\PU(2,1)$.

\section{A Parametrisation of $[m_1,m_2,0]$-Triangle Groups}

\label{sec-param}

\myskip\noindent
For $r_1,r_2\ge1$ and $\al\in(0,2\pi)$ let $C_1$, $C_2$ and $C_3$ be the complex geodesics with respective normalised polar vectors
$$
c_1=\left(\begin{matrix} \sqrt{2}r_2e^{-i\theta} \\ 1 \\ 0 \end{matrix}\right), \quad
c_2=\left(\begin{matrix} -\sqrt{2}r_1e^{i\theta} \\ 1 \\ 0 \end{matrix}\right), \quad
c_3=\left(\begin{matrix} 1/\sqrt{2} \\ 0 \\ 1/\sqrt{2} \end{matrix}\right).
$$
where $\theta=(\pi-\al)/2\in(-\pi/2,\pi/2)$.
The type of the triangle formed by $C_1$, $C_2$ and $C_3$ is determined by 
$$
 |\<c_3, c_2\>|=r_1,\quad
 |\<c_1,c_3\>|=r_2,\quad
 |\<c_2,c_1\>|=1
$$
and the angular invariant
$$
  \arg\left(\prod^{3}_{k=1}\<c_{k-1},c_{k+1}\>\right)
  =\arg(-r_1r_2 e^{-2i\theta})
  =\pi-2\theta
  =\al.
$$
The triangle formed by $C_1$, $C_2$ and $C_3$ is then an ultra-parallel $[m_1,m_2,0]$-triangle with angular invariant~$\al$,
where $\cosh(m_j/2)=r_j$ for $j=1,2$.

\myskip
Every value of the angular invariant between~$0$ and~$2\pi$ 
and hence each isometry type of $[m_1,m_2,0]$-triangles
is represented among the parametrisations of the form given above.

\myskip
Let $I_j$ denote the inversion in the chain $C_j$, $j=1,2,3$.
Let $\Ga=\<I_1,I_2,I_3\>$ be the group generated by $I_1$, $I_2$ and $I_3$
and let $\Ga'=\<I_1,I_2\>$ be the group generated by just $I_1$ and~$I_2$.

\myskip
We shall now revert from looking at reflections in the geodesics $C_1$, $C_2$ and $C_3$
and instead talk about reflections in the corresponding chains, which we denote by $C_1$, $C_2$ and $C_3$ as well.
If we look at the arrangement of the chains $C_1$, $C_2$ and $C_3$ in $\calN$,
the finite chain $C_3$ is the (Euclidean) unit circle in $\c\times\{0\}$,
whereas $C_1$ and $C_2$ are vertical lines through $r_2e^{i\theta}$ and $-r_1e^{-i\theta}$ respectively,
see Figure~\ref{fig-chains}.
Since $r_1,r_2>1$, the chains $C_1$ and $C_2$ lie outside the chain $C_3$.


\begin{figure}[h]
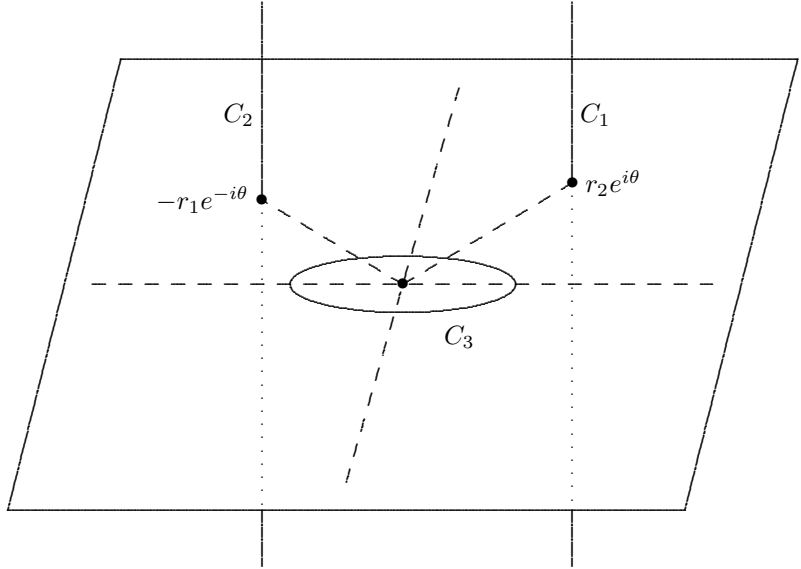

\begin{center}
\leavevmode
  \setcoordinatesystem units <0.75cm,0.75cm> point at 0 0
  \setplotarea x from -8 to 8, y from -5 to 5
  \plot -7 -4 5 -4 7 4 -5 4 -7 -4 /
  \ellipticalarc axes ratio 4:1 360 degrees from 2 0 center at 0 0
  \plot -2.5 1.5 -2.5 5 /
  \plot 3 1.8 3 5 /
  \plot -2.5 -4 -2.5 -5 /
  \plot 3 -4 3 -5 /
  \put {$C_3$} [t] <0pt,-5pt> at 1 -0.5
  \put {$C_2$} [r] <-3pt,0pt> at -2.5 3
  \put {$C_1$} [l] <3pt,0pt> at 3 3
  \put {$r_2 e^{i\theta}$} [l] <5pt,0pt> at 3 1.8
  \put {$-r_1 e^{-i\theta}$} [r] <-5pt,0pt> at -2.5 1.5
  \multiput {$\bullet$} at 0 0 -2.5 1.5 3 1.8 /
  \setdashes
  \plot 0 0 -2.5 1.5 /
  \plot 0 0 3 1.8 /
  \plot -5.5 0 5.5 0 /
  \plot -1 -3.5 1 3.5 /
  \setdots
  \plot -2.5 1.5 -2.5 -4 /
  \plot 3 1.8 3 -4 /
\end{center}
\caption{Chains $C_1$, $C_2$ and $C_3$}
\label{fig-chains}
\end{figure} 

\myskip
Recall that inversion~$I_j$ in the vertical chain~$C_j$
rotates any vertical chain as a set around~$C_j$ through~$\pi$.
Let $J_1$ and $J_2$ be the rotations of~$\c\times\{0\}$
around~$r_2e^{i\theta}$ and $-r_1e^{-i\theta}$ through~$\pi$ respectively.
Let $\La=\<J_1,J_2\>$ be the group of isometries of~$\c\times\{0\}$ generated by~$J_1$ and~$J_2$.
The rotations~$J_1$ and $J_2$ are of order~$2$,
so we can represent any element of~$\La$ as an alternating product of~$J_1$ and~$J_2$:
$$\La=\{(J_2J_1)^{\ell},J_1(J_2J_1)^{\ell}\st\,\ell\in\z\}.$$
We compute
\begin{align*}
  &(J_2J_1)^{\ell}(0)=-2\ell(r_2e^{i\theta}+r_1e^{-i\theta}),\\
  &J_1(J_2J_1)^{\ell}(0)=2r_2e^{i\theta}+2\ell(r_2e^{i\theta}+r_1e^{-i\theta}).
\end{align*}
Note that the projection $\Ga'\to\La$ given by $I_1\mapsto J_1$ and $I_2\mapsto J_2$ is injective.

\section{Discreteness Results}

\label{sec-discr-results}

\myskip\noindent
We will need the following lemma:

\begin{lem}
\label{lem1}
If $|g(0)|\ge2$ for each~$g\in\Lastar$, then the group $\Ga$ is discrete.
\end{lem}

\begin{proof}
We will use the discreteness criterion described above. 
Consider the unit spinal sphere
$$U=\{(\zeta,\upsilon)\in\calN:|\zeta|^4+\upsilon^2=1\}.$$
The inversion $I_3$ in $C_3$ is given by
$$I_3([z_1:z_2:z_3])=[z_3:-z_2:z_1]\qquad\text{for}\quad[z_1:z_2:z_3]\in\chp.$$
The inversion~$I_3$ preserves the bisector
$$\mathfrak{C}=\bigl\{[z_1:z_2:z_3]\in\chp\st|z_1|=|z_3|\}$$
and hence preserves the unit spinal sphere~$U$, which is the boundary of~$\mathfrak{C}$.
The inversion~$I_3$ interchanges the points $[0:0:1]$ and $[1:0:0]$ in~$\chp$,
which correspond to the points $(0,0)\in\c\times\r$ and infinity in~$\calN$.
Hence $I_3$ leaves $U$ invariant and switches the inside of $U$ with the outside.

\myskip
Let $U_1$ be the part of $\calN\smallsetminus U$ outside~$U$, containing $\infty$,
and let $U_2$ be the part inside~$U$, containing the origin.
Clearly we have $U_1\cap U_2=\emptyset$ and $I_3(U_1)=U_2$.
Therefore if we find a subset $V\varsubsetneq U_1 $ such that $g(U_2)\varsubsetneq V$ for all elements $g\in\Gaprimstar$,
then we have proved that $\Ga$ is compressing and hence discrete.
Let
$$W=\{(\zeta,\upsilon)\in\calN:|\zeta|=1\}$$
be the set of all vertical chains through $\zeta\in\c$ with $|\zeta|=1$.
Let
$$
  W_1=\{(\zeta,\upsilon)\in\calN:|\zeta|>1\}
  \quad\text{and}\quad
  W_2=\{(\zeta,\upsilon)\in\calN:|\zeta|<1\}.
$$
We have $U_2\subset W_2$ and so $g(U_2)\subset g(W_2)$ for all elements~$g\in\Gaprimstar$.
The set~$W_2$ is a union of vertical chains. 
Elements of~$\Ga'$ map vertical chains to vertical chains.
Therefore we can simply look at the intersection of the images of~$W_2$ with $\c\times\{0\}$.
For each~$g$ the image~$g(W_2)$ of~$W_2$ intersects $\c\times\{0\}$ in a disc.
Elements of $\Ga'$ move the intersection with $\c\times\{0\}$ 
by rotations $J_1$ and $J_2$ around $r_2e^{i\theta}$ and $-r_1e^{-i\theta}$ through~$\pi$.
The projection $\Ga'\to\La$ is injective, hence elements of $\Gaprimstar$ move the intersection with $\c\times\{0\}$ by elements of $\Lastar$.
Provided that the interior of the unit circle is mapped completely off itself under all elements in~$\Lastar$,
then the same is true for $W_2$ and hence for $U_2$ under~$\Ga'$.
We can then choose $V$ to be the union of all the images of $U_2$.
We can be sure that $V\ne U_1$ since $V$ is missing all the images of $W_2\smallsetminus U_2$.
We are therefore only left to find what is required
to be sure that the interior of the unit circle is mapped off itself by any element in~$\Lastar$.
Since the radius of a circle is preserved under rotations,
it suffices to show that the origin is moved to a distance of at least~$2$ by any element in~$\Lastar$.
This is precisely the condition of the lemma.
\end{proof}

\begin{lem}
\label{lem-*}
The condition $|g(0)|\ge2$ holds for all elements~$g\in\Lastar$ if and only if
$$a(\ell)\ge1\quad\text{for all}~\ell\in\z\setminus\{-1,0\}\quad\text{and}\quad b\ge1\qquad{(*)},$$
where
$$a(\ell)=|r_2e^{i\theta}+\ell(r_2e^{i\theta}+r_1e^{-i\theta})|,\quad b=|r_2e^{i\theta}+r_1e^{-i\theta}|.$$
\end{lem}

\begin{proof}
Recall that
$$\La=\{(J_2J_1)^{\ell},J_1(J_2J_1)^{\ell}\st\,\ell\in\z\}$$
and
\begin{align*}
  &(J_2J_1)^{\ell}(0)=-2\ell(r_2e^{i\theta}+r_1e^{-i\theta}),\\
  &J_1(J_2J_1)^{\ell}(0)=2r_2e^{i\theta}+2\ell(r_2e^{i\theta}+r_1e^{-i\theta}).
\end{align*}
We need $|g(0)|\ge2$ for all $g\in\Lastar$, i.e.\
\begin{align*}
  a(\ell)&=|r_2e^{i\theta}+\ell(r_2e^{i\theta}+r_1e^{-i\theta})|\ge1\quad\text{for all}~\ell\in\z,\\
  |\ell|\cdot b&=|\ell(r_2e^{i\theta}+r_1e^{-i\theta})|\ge1\quad\text{for all}~\ell\in\z\setminus\{0\}.
\end{align*}
Note that it is sufficient to check the inequality $a(\ell)\ge 1$ for~$\ell\in\z\setminus\{-1,0\}$
as it is always satisfied for $a(-1)=r_1$ and $a(0)=r_2$.
Also note that it is sufficient to only check that $|\ell|\cdot b\ge 1$ for~$\ell=1$
as $b\ge1$ implies that the inequality holds for all~$\ell\in\z\setminus\{0\}$.
\end{proof}

\myskip
We will start with a rough estimate on~$a(\ell)$ and~$b$ obtained by taking $|\Re(z)|$ as a lower bound for $|z|$:

\begin{lem}
\label{lem3}
$$\sin\left(\frac{\al}{2}\right)\ge\frac{1}{r_1+r_2}$$
implies conditions~$(*)$.
\end{lem}

\begin{proof}
Recall that $\sin(\al/2)=\cos(\theta)$.
Note that $\theta\in(-\pi/2,\pi/2)$ and hence $\cos\theta>0$.
Suppose $\cos\theta\ge\frac{1}{r_1+r_2}$. 
Then
\begin{align*}
  b&=|r_2e^{i\theta}+r_1e^{-i\theta}|
  \ge|\Re(r_2e^{i\theta}+r_1e^{-i\theta})|
  =(r_1+r_2)\cdot\cos\theta
  \ge1,\\
  a(\ell)&=|r_2e^{i\theta}+\ell(r_2e^{i\theta}+r_1e^{-i\theta})|
  \ge|\Re(r_2e^{i\theta}+\ell(r_2e^{i\theta}+r_1e^{-i\theta}))|\\
  &=|r_2+\ell(r_1+r_2)|\cdot\cos\theta
  =|\ell r_1+(\ell+1)r_2|\cdot\cos\theta\\
  &\ge(r_1+r_2)\cdot\cos\theta
  \ge1\quad\text{for}~\ell\in\z\backslash\{-1,0\}.\qedhere
\end{align*}
\end{proof}

\myskip\noindent
Combining the results of Lemmas~\ref{lem1}, \ref{lem-*} and~\ref{lem3} we obtain Proposition~\ref{conditions1}.
We will now calculate~$a(\ell)$ and $b$ to obtain more refined estimates:

\begin{lem}
\label{lem4new}
Conditions~$(*)$ are equivalent to
\begin{align*}
  &4r_1r_2\sin^2\left(\frac{\al}{2}\right)\ge\fb,\\
  &4r_1r_2\sin^2\left(\frac{\al}{2}\right)\ge\fa(\ell)\quad\text{for all}~\ell\in\z\setminus\{-1,0\},
\end{align*}
where
$$\fa(\ell)=\frac{1-(\ell r_1-(\ell+1)r_2)^2}{\ell(\ell+1)},\quad \fb=1-(r_1-r_2)^2.$$
\end{lem}

\begin{proof}
Conditions $(*)$ state that $b\ge1$ and $a(\ell)\ge 1$ for all $\ell\in\z\setminus\{-1,0\}$, where
\begin{align*}
  b^2&=|r_2e^{i\theta}+r_1e^{-i\theta}|^2 \\
  &=r_2^2+2r_1r_2\cos(2\theta)+r_1^2 \\
  &=(r_1-r_2)^2+4r_1r_2\cos^2(\theta),\\
  a^2(\ell)
  &=|r_2e^{i\theta}+\ell(r_2e^{i\theta}+r_1e^{-i\theta})|^2 \\
  &=(\ell+1)^2r_2^2+2\ell(\ell+1)r_1r_2\cos(2\theta)+\ell^2r_1^2 \\
  &=(\ell r_1-(\ell+1)r_2)^2+4\ell(\ell+1)r_1r_2\cos^2(\theta).
  \end{align*}
Rearranging these expressions to give an inequality in $\cos^2(\theta)$ and using $\cos(\theta)=\sin(\al/2)$ gives the result.
\end{proof}

\myskip\noindent
We will now discuss some properties of~$\fa$ as a function of~$\ell$:

\begin{lem}
\label{lem-about-fA}
Consider the function $$\fa(\ell)=\frac{1-(\ell r_1-(\ell+1)r_2)^2}{\ell(\ell+1)}\quad\text{for}~\ell\in\z\setminus\{-1,0\}.$$
\begin{enumerate}[(a)]
\item
Suppose that
$$r_1^2-1\ge\frac{k+2}{k}(r_2^2-1)\quad\text{for some integer}~k\ge1,$$
then $\fa(\ell_1)\ge\fa(\ell_2)$ for all integers $\ell_1,\ell_2$ with either $k\le\ell_1<\ell_2$ or $\ell_2\le -2$, $\ell_1\ge k/2$.
\item
Suppose that
$$r_1^2-1\le\frac{k+1}{k-1}(r_2^2-1)\quad\text{for some integer}~k\ge2,$$
then $\fa(\ell_1)\ge\fa(\ell_2)$ for all integers $\ell_1,\ell_2$ with $1\le\ell_2<\ell_1\le k$. 
\end{enumerate}
\end{lem}

\begin{proof}
The function $\fa$ can be rewritten as
$$\fa(\ell)=\frac{1-(\ell r_1-(\ell+1)r_2)^2}{\ell(\ell+1)}=\frac{r_1^2-1}{\ell+1}-\frac{r_2^2-1}{\ell}-(r_1-r_2)^2.$$
For any $\ell_1,\ell_2\in\z\setminus\{-1,0\}$ we have
$$\fa(\ell_1)-\fa(\ell_2)=(\ell_1-\ell_2)\left(\frac{r_2^2-1}{\ell_1\ell_2}-\frac{r_1^2-1}{(\ell_1+1)(\ell_2+1)}\right).$$

\myskip\noindent
Suppose that $r_2=1$, then
$$\fa(\ell_1)-\fa(\ell_2)=-\frac{(\ell_1-\ell_2)(r_1^2-1)}{(\ell_1+1)(\ell_2+1)}.$$
In this case the condition
$$r_1^2-1\ge3(r_2^2-1)$$
is satisfied.
This corresponds to part~(a) with~$k=1$. 
Then for $1\le\ell_1<\ell_2$ we have $\fa(\ell_1)-\fa(\ell_2)\ge0$ since
$$\ell_1-\ell_2<0,\quad \ell_1+1,\ell_2+1>0.$$
For $\ell_2\le-2$, $\ell_1\ge1$ we have $\fa(\ell_1)-\fa(\ell_2)\ge0$ since
$$\ell_1-\ell_2,\ell_1+1>0,\quad\ell_2+1<0.$$

\myskip\noindent
Now suppose that $r_2\ne1$, then we can rewrite $\fa(\ell_1)-\fa(\ell_2)$ as
\begin{align*}
  \fa(\ell_1)-\fa(\ell_2)
  &=\frac{(\ell_1-\ell_2)(r_2^2-1)}{(\ell_1+1)(\ell_2+1)}
       \left(\frac{1}{\ell_1}+\frac{1}{\ell_2}+\frac{1}{\ell_1\ell_2}-\frac{r_1^2-r_2^2}{r_2^2-1}\right).
\end{align*}
Suppose, for some integer $k\ge 1$, that
$$r_1^2-1\ge\frac{k+2}{k}(r_2^2-1),\quad\text{i.e.}\quad\frac{r_1^2-r_2^2}{r_2^2-1}\ge\frac{2}{k}.$$
Then for $1\le k\le\ell_1<\ell_2$ we have $\fa(\ell_1)-\fa(\ell_2)\ge0$ since $\ell_1-\ell_2<0$, $\ell_1+1,\ell_2+1>0$ and
$$\frac{1}{\ell_1}+\frac{1}{\ell_2}+\frac{1}{\ell_1\ell_2}-\frac{r_1^2-r_2^2}{r_2^2-1}\le\frac{1}{k}+\frac{1}{k+1}+\frac{1}{k(k+1)}-\frac{2}{k}=0.$$
For $\ell_2\le -2$, $\ell_1\ge k/2$ we have $\fa(\ell_1)-\fa(\ell_2)\ge0$ since $\ell_1-\ell_2,\ell_1+1>0$, $\ell_2+1<0$ and 
$$\frac{1}{\ell_1}+\frac{1}{\ell_2}+\frac{1}{\ell_1\ell_2}-\frac{r_1^2-r_2^2}{r_2^2-1}\le\frac{2}{k}-\frac{2}{k}\le0.$$
Suppose, for some integer $k\ge 2$, that
$$r_1^2-1\le\frac{k+1}{k-1}(r_2^2-1),\quad\text{i.e.}\quad\frac{r_1^2-r_2^2}{r_2^2-1}\le\frac{2}{k-1}.$$
Then for $1\le\ell_2<\ell_1\le k$ we have $\fa(\ell_1)-\fa(\ell_2)\ge0$ since $\ell_1-\ell_2,\ell_1+1,\ell_2+1>0$ and
$$
  \frac{1}{\ell_1}+\frac{1}{\ell_2}+\frac{1}{\ell_1\ell_2}-\frac{r_1^2-r_2^2}{r_2^2-1}
  \ge\frac{1}{k}+\frac{1}{k-1}+\frac{1}{k(k-1)}-\frac{2}{k-1}=0.\qedhere
$$
\end{proof}

\myskip\noindent
We will now discuss for which~$\ell$, depending on~$r_1$ and~$r_2$, does the inequality 
$$4r_1r_2\sin^2\left(\frac{\al}{2}\right)\ge\fa(\ell)$$
in Lemma~\ref{lem4new} give the strongest estimate on $\sin^2(\al/2)$:

\begin{lem}
\label{lem5new}
Suppose $m_1\ge m_2\ge0$.
Conditions~$(*)$ hold if one of the following conditions is satisfied:
\begin{enumerate}
\item[(i)] 
$3(r_2^2-1)\le r_1^2-1$ and $4r_1r_2\sin^2(\al/2)\ge\max\{\fb,\fa(1)\}$.
\item[(ii)]
For some integer~$k\ge 2$
$$\frac{k+2}{k}(r_2^2-1)\le r_1^2-1\le\frac{k+1}{k-1}(r_2^2-1)$$ 
and
$$4r_1r_2\sin^2(\al/2)\ge\max\{\fb,\fa(k)\}.$$
\end{enumerate}
\end{lem}

\begin{proof}
Note that part (i) is identical to part (ii) with $k=1$ except that there is no upper bound on $r_1^2-1$.
Setting $\ell_1=k$ in Lemma~\ref{lem-about-fA} we obtain that $\fa(k)\ge\fa(\ell)$ for all $\ell\in\z\setminus\{0,\,-1\}$.
This means that our hypothesis $4r_1r_2\sin^2(\al/2)\ge\fa(k)$ implies $4r_1r_2\sin^2(\al/2)\ge\fa(\ell)$ for all $\ell\in\z\setminus\{0,\,-1\}$.
This proves the result.
\end{proof}

\myskip\noindent
It remains to decide which of~$\fb$ and~$\fa(k)$, depending on~$r_1$ and~$r_2$,
is the stronger estimate on $4r_1r_2\sin^2(\al/2)$:

\begin{lem}
\label{lem6new}
Suppose $m_1\ge m_2\ge0$.
Conditions~$(*)$ hold if one of the following conditions is satisfied:
\begin{enumerate}[(a)]
\item
$$\max\left\{2r_2^2,\ 3(r_2^2-1)\right\}\le r_1^2-1\quad\text{and}\quad 4r_1r_2\sin^2(\al/2)\ge\fa(1).$$
\item
For some integer~$k\ge2$
$$\max\left\{\frac{k+1}{k}(r_2^2-1)+(k+1),\frac{k+2}{k}(r_2^2-1)\right\}\le r_1^2-1\le \frac{k+1}{k-1}(r_2^2-1)$$
and
$$4r_1r_2\sin^2(\al/2)\ge\fa(k).$$
\item
For all integers~$\ell\ge1$
$$r_1^2-1\le \frac{\ell+1}{\ell}(r_2^2-1)+(\ell+1)\quad\text{and}\quad 4r_1r_2\sin^2(\al/2)\ge\fb.$$
\end{enumerate}
\end{lem}

\begin{rem}
Figure~\ref{fig-regions1} shows the regions in parts~(a)--(c) of Lemma~\ref{lem6new} in the case $r_2\ne1$
in the coordinates 
$$
  (X,Y)
  =\left(\frac{r_1^2-r_2^2}{r_2^2-1},\frac{1}{r_2^2-1}\right)
  =\left(\frac{\cosh^2(m_1/2)-1}{\cosh^2(m_2/2)-1}-1,\frac{1}{\cosh^2(m_2/2)-1}\right).
$$
Part~(a) corresponds to the part of the shaded region with $X\ge2$.
Part~(b) for $k\ge2$ corresponds to the part of the shaded region with $\frac{2}{k}\le X\le\frac{2}{k-1}$.
Part~(c) corresponds to the unshaded region above the broken line.
\end{rem}

\myskip\noindent
Finally we compute traces of certain elements in the group to rephrase conditions~$(*)$ in Lemmas~\ref{lem-*} and~\ref{lem4new} in terms of these traces and in terms of ellipticity of these elements.

\begin{lem}
\label{lem-traces}
The traces of the elements
$$\wal=I_1(I_2I_1)^{\ell}I_3\quad\text{and}\quad \wb=I_1I_2I_3$$
are
\begin{align*}
  \trace(\wal)
  &=4\big|\ell r_1e^{i\theta}+(\ell+1)r_2e^{-i\theta}\big|^2-1\\
  &=4(\ell r_1-(\ell+1)r_2)^2-1+16\ell(\ell+1)r_1r_2\sin^2(\al/2),\\
  \trace(\wb)
  &=-4|r_2e^{i\theta}+r_1e^{-i\theta}|^2-1+i\cdot 8r_1r_2\sin(2\theta)\\
  &=-(4r_1^2+4r_2^2+1)+8r_1r_2\cdot e^{i\al}\\
  &=-4(r_1-r_2)^2-1-16r_1r_2\sin^2(\al/2)+i\cdot 16r_1r_2\sin(\al/2)\cos(\al/2),\\
\end{align*}
therefore conditions~$(*)$ in Lemmas~\ref{lem-*} and~\ref{lem4new} are equivalent to
$$\Re(\trace(\wb))\le -5\quad\text{and}\quad\trace(\wal)\ge3~\text{for all}~\ell\in\z\setminus\{-1,0\}.$$
Moreover, $\trace(\wal)\ge3$ is equivalent to~$\wal$ being not regular elliptic,
while $\Re(\trace(\wb))\le -5$ implies that $\wb$ is non-elliptic (but is not equivalent to it).
\end{lem}

\begin{proof}
The computations of the traces are straightforward.
The ellipticity of the elements~$\wal$ and~$\wb$ 
can be determined by looking at the position of their traces in the complex plane in relation to the deltoid~$\De$ as explained in section~\ref{sec-basics}.
The traces of the elements~$\wal$ are real.
The portion of the real axis within the deltoid~$\De$ is $(-1,3)$.
It is easy to see that $\trace(\wal)\ge-1$, hence $\trace(\wal)\ge3$ is equivalent to~$\wal$ being not regular elliptic.
The condition $\Re(\trace(\wb))\le -5$ implies that $\wb$ is non-elliptic (but is not equivalent to it).
\end{proof}

\myskip\noindent
Combining the results of Lemmas~\ref{lem1}, \ref{lem-*}, \ref{lem4new}, \ref{lem6new} and \ref{lem-traces} we obtain Propositions~\ref{conditions2-typeA} and~\ref{conditions2-typeB}.

\section{Discreteness for all Values of the Angular Invariant}

\label{sec-discr-for-all-alpha}

\myskip\noindent
Let $\Ga(\al)$ be a complex hyperbolic $[m_1,m_2,0]$-triangle groups with the angular invariant~$\al$.
Proposition~\ref{discrete} states that for some choices of~$m_1$ and~$m_2$ the discreteness conditions in Lemma~\ref{lem4new} allow us to show that $\Ga(\al)$ is discrete for all values of~$\al$.

\myskip\noindent
{\bf Proof of Proposition~\ref{discrete}:}
The conditions 
$$r_1-r_2\in\left(\bigcup\limits_{k=2}^{\infty}\left[\frac{r_2+1}{k},\frac{r_2-1}{k-1}\right]\right)\cup\left[r_2+1,\infty\right)$$
of Proposition~\ref{discrete} can be rewritten as
$$
  r_1\ge 2r_2+1,
  \quad\text{or}\quad 
  \frac{(k+1)r_2+1}{k}\le r_1\le \frac{kr_2-1}{k-1}
$$
for some integer $k\ge 2$.
(Note that the latter condition can only hold for $r_2\ge 2k-1$.)
The corresponding regions (in the coordinates $x=\frac{r_1-r_2}{r_2}=\frac{r_1}{r_2}-1$, $y=\frac{1}{r_2}$) are the shaded areas in Figure~\ref{fig-regions2}.

\begin{figure}[h]
  \begin{center}
    \forcehmode
      \bgroup
        \beginpicture
          \setcoordinatesystem units <5cm,5cm>
          \setplotarea x from 0 to 2.1, y from 0 to 1.2
          \arrow <7pt> [0.2,0.5] from 0 0 to 0 1.2
          \arrow <7pt> [0.2,0.5] from 0 0 to 2.1 0      
          \plot 0 1 2.1 1 /
          \setshadegrid span <2pt>
          \plot 0.2 0 0.2222 0.1111 0.25 0 0.2857 0.1428 0.3333 0 0.4 0.2 0.5 0 0.6666 0.3333 1 0 2 1 2.1 1 /
          \vshade 0.2 0 0 0.2222 0 0.1111 0.25 0 0 0.2857 0 0.1428 0.3333 0 0 0.4 0 0.2 0.5 0 0 0.6666 0 0.3333 1 0 0 2 0 1 2.1 0 1 /
          \plot 0.125 0 0.1333 0.0666 0.1428 0 0.1538 0.0769 0.1666 0 0.1818 0.0909 0.2 0 /
          \vshade 0.125 0 0 0.1333 0 0.0666 0.1428 0 0 0.1538 0 0.0769 0.1666 0 0 0.1818 0 0.0909 0.2 0 0 /
          \setdots
          \plot 0 1 1 0 2 1 /
          \plot 0 1 0.5 0 1 1 /
          \plot 0 1 0.333 0 0.666 1 /
          \plot 0 1 0.25 0 0.5 1 /
          \multiput {$\bullet$} at 0 0 0.125 0 0.1428 0 0.1666 0 0.2 0 0.25 0 0.3333 0 0.5 0 1 0 /   
          \multiput {$\bullet$} at 0.1333 0.0666 0.1538 0.0769 0.1818 0.0909 0.2222 0.1111 0.2857 0.1429 0.4 0.2 0.6666 0.3333 2 1 /   
          \put {$0$} [tr] <-3pt,-3pt> at 0 0
          \put {$\frac{1}{1}$} [t] <0pt,-5pt> at 1 0
          \put {$\frac{1}{2}$} [t] <0pt,-5pt> at 0.5 0
          \put {$\frac{1}{3}$} [t] <0pt,-5pt> at 0.3333 0
          \put {$\frac{1}{4}$} [t] <0pt,-5pt> at 0.25 0
          \put {$\frac{1}{5}$} [t] <0pt,-5pt> at 0.2 0
          \put {$\cdots$} [t] <0pt,-5pt> at 0.1 0
          \put {$(\frac{2}{3},\frac{1}{3})$} [l] <5pt,0pt> at 0.6666 0.3333
          \put {$(\frac{2}{5},\frac{1}{5})$} [l] <5pt,0pt> at 0.4 0.2
          \put {$x=\frac{r_1}{r_2}-1$} [tr] <0pt,-3pt> at 2.1 0
          \put {$y=\frac{1}{r_2}$} [l] <3pt,0pt> at 0 1.2
          \put {$r_2=1$} [b] <0pt,3pt> at 1.5 1
          \put {$r_1-2r_2=1$} [l] <5pt,0pt> at 1.25 0.25
          \put {$r_1-2r_2=-1$} [l] <5pt,0pt> at 0.75 0.25
        \endpicture
      \egroup
  \end{center}
  \caption{Conditions of Proposition~\ref{discrete}}
  \label{fig-regions2}
\end{figure}

\noindent
First suppose  $r_1\ge 2r_2+1$.
Then
$$r_1-r_2\ge r_2+1\ge 2.$$
If $\ell\ge 1$ then
$$\ell r_1-(\ell+1)r_2 \ge (\ell-1)r_2+\ell \ge 1.$$
If $\ell\le 0$
$$\ell r_1-(\ell+1)r_2 \le (\ell-1)r_2+\ell \le -1.$$
Therefore $\fb,\fa(\ell)\le0$ and the conditions on $4r_1r_2\sin^2(\al/2)$ in Lemma \ref{lem4new} are always satisfied.

Now suppose $\frac{(k+1)r_2+1}{k}\le r_1\le \frac{kr_2-1}{k-1}$ for some 
integer $k\ge 2$ and note that $r_2\ge 2k-1$.
Then
$$r_1-r_2\ge \frac{r_2+1}{k}\ge 2.$$
If $\ell\ge k$ then
$$\ell r_1-(\ell+1)r_2 \ge \frac{(\ell-k)r_2+\ell}{k} \ge 2(\ell-k)+1\ge 1.$$
If $0\le\ell\le k-1$ then 
$$\ell r_1-(\ell+1)r_2 \le \frac{-(k-\ell-1)r_2-\ell}{k-1}\le -2(k-\ell-1)-1\le -1.$$
If $\ell\le 0$ then
$$\ell r_1-(\ell+1)r_2 \le \frac{(\ell-k)r_2+\ell}{k} \le 2\ell-2(k-1)-1\le -1.$$
Therefore $\fb,\fa(\ell)\le0$ and the conditions on $4r_1r_2\sin^2(\al/2)$ in Lemma \ref{lem4new} are always satisfied.\qed

\myskip\noindent
We now compare the conditions on $r_1$ and $r_2$ coming from Lemma \ref{lem6new} and from Proposition~\ref{discrete}.

\begin{prop}
\label{hyper-Phi-k}
Suppose $m_1\ge m_2\ge0$.
In the case $m_2=0$ any complex hyperbolic ultra-parallel $[m_1,0,0]$-triangle group is discrete if $r_1\ge3$.
Now suppose $m_2>0$.
For each positive integer $k$ define 
$$\Phi_k(X)=\frac{(k^2X-2k-1)^2}{4k(k+1)(kX-1)}.$$
\begin{enumerate}
\item[(a)] If 
$$
\frac{r_1^2-r_2^2}{r_2^2-1}\ge 2 \quad \hbox{ and }\quad
\frac{1}{r_2^2-1}\le \Phi_1\left(\frac{r_1^2-r_2^2}{r_2^2-1}\right)
$$
then any complex hyperbolic ultra-parallel $[m_1,m_2,0]$-triangle group is discrete.
\item[(b)] If there is an integer $k\ge 2$ so that
$$
\frac{2}{k}\le \frac{r_1^2-r_2^2}{r_2^2-1}\le \frac{2}{k-1} \quad \hbox{\and }\quad
\frac{1}{r_2^2-1}\le \Phi_k\left(\frac{r_1^2-r_2^2}{r_2^2-1}\right)
$$
then any complex hyperbolic ultra-parallel $[m_1,m_2,0]$-triangle group is discrete.
\end{enumerate}
\end{prop}

\begin{proof}
In the case $m_2=0$ we have~$r_2=1$, so that the conditions of part~(b) of Lemma~\ref{lem6new} are never satisfied, while the conditions of parts~(a) and~(c) can be rewritten as
$$r_1\ge\sqrt{3}\quad\text{and}\quad\sin^2\left(\frac{\al}{2}\right)\ge\frac{1-(r_1-2)^2}{8r_1}$$
or
$$r_1\le\sqrt{3}\quad\text{and}\quad\sin^2\left(\frac{\al}{2}\right)\ge\frac{1-(r_1-1)^2}{4r_1}.$$
For $r_1\ge3$ we use part~(a) of Lemma~\ref{lem6new} to see that the conditions are satisfied for all values of~$\al$.
Proposition~\ref{discrete} in the case $r_2=1$ gives the same condition $r_1\ge3$.

\myskip
Now suppose $m_2>0$, i.e.\ $r_2>1$.
Note that the statement in part (a) is the same as the statement in (b) with $k=1$
but where there is no upper bound on $(r_1^2-r_2^2)/(r_2^2-1)$.
The conditions on~$r_1^2-1$ in part~(b) of Lemma~\ref{lem6new} can be rewritten as
$$\max\left\{\frac{1}{k}+\frac{k+1}{r_2^2-1},\frac{2}{k}\right\}\le\frac{r_1^2-r_2^2}{r_2^2-1}\le\frac{2}{k-1}$$
for~$k\ge2$ and the same but without the upper bound on $(r_1^2-r_2^2)/(r_2^2-1)$ for~$k=1$.
Note that the function 
$$\Phi_k(X)=\frac{(k^2X-2k-1)^2}{4k(k+1)(kX-1)}=\frac{1}{4}\left(\frac{k(kX-1)}{k+1}+\frac{k+1}{k(kX-1)}\right)-\frac{1}{2}$$
defines a hyperbola with vertical asymptote $X=\tfrac{1}{k}$, tangent to the $X$ axis at $X=\tfrac{2k+1}{k^2}$
with values
$$\Phi_k\left(\frac{2}{k-1}\right)=\frac{1}{4(k-1)k},\quad \Phi_k\left(\frac{2}{k}\right)=\frac{1}{4k(k+1)}.$$
It is not hard to show that when $k\ge2$ and $\tfrac{2}{k}\le X\le\tfrac{2}{k-1}$
$$\Phi_k(X)\le\frac{kX-1}{4k(k+1)}<\frac{kX-1}{k(k+1)},$$
hence $\tfrac{1}{r_2^2-1}\le\Phi_k(X)$ implies
$$\frac{1}{r_2^2-1}\le\Phi_k(X)\le\frac{kX-1}{4k(k+1)}<\frac{kX-1}{k(k+1)}$$
and therefore
$$\frac{1}{k}+\frac{k+1}{r_2^2-1}\le X.$$

\myskip\noindent
To summarise, 
$$
   \frac{2}{k}\le \frac{r_1^2-r_2^2}{r_2^2-1}\le \frac{2}{k-1} \quad \hbox{\and }\quad
   \frac{1}{r_2^2-1}\le \Phi_k\left(\frac{r_1^2-r_2^2}{r_2^2-1}\right)
$$
implies
$$\max\left\{\frac{1}{k}+\frac{k+1}{r_2^2-1},\frac{2}{k}\right\}\le\frac{r_1^2-r_2^2}{r_2^2-1}\le\frac{2}{k-1}.$$
Thus, using Lemma~\ref{lem6new}, if we can show that
$$\sin^2(\alpha/2)\ge \frac{1-\bigl(kr_1-(k+1)r_2\bigr)^2}{4k(k+1)r_1r_2}$$
then conditions $(*)$ hold, and so the group is discrete.

\myskip\noindent
The condition 
$$\frac{1}{r_2^2-1}\le \Phi_k\left(\frac{r_1^2-r_2^2}{r_2^2-1}\right)$$
is equivalent to
\begin{eqnarray*}
  0
  &\le&\Bigl(k^2(r_1^2-1)-(k+1)^2(r_2^2-1)\Bigr)^2-4k(k+1)\Bigl(k(r_1^2-1)-(k+1)(r_2^2-1)\Bigr) \\
  &=&\bigl(k^2r_1^2-(k+1)^2r_2^2\bigr)^2-2k^2r_1^2-2(k+1)^2r_2^2+1 \\
  &=&\Bigl(\bigl(kr_1+(k+1)r_2\bigr)^2-1\Bigr)\Bigl(\bigl(kr_1-(k+1)r_2\bigr)^2-1\Bigr).
\end{eqnarray*}
In particular, we have $\bigl(kr_1-(k+1)r_2\bigr)^2\ge 1$.
Hence the conditions from Lemma~\ref{lem6new} are satisfied for all values of~$\al$.
\end{proof}

\begin{rem}
Figure~\ref{fig-regions3} shows the regions in parts (a)--(b)  of Proposition~\ref{hyper-Phi-k}
in the case $r_2\ne1$ in the coordinates 
$$
  (X,Y)
  =\left(\frac{r_1^2-r_2^2}{r_2^2-1},\frac{1}{r_2^2-1}\right)
  =\left(\frac{\cosh^2(m_1/2)-1}{\cosh^2(m_2/2)-1}-1,\frac{1}{\cosh^2(m_2/2)-1}\right).
$$
The lightly shaded region is as in Figure~\ref{fig-regions1}.
The values of $(m_1,m_2)$ in Proposition~\ref{hyper-Phi-k} correspond
to the darkly shaded regions under hyperbolae in Figure~\ref{fig-regions3}.
Only hyperbolae~$\Phi_k$ for $1\le k\le5$ are shown, but the dark shaded regions continue to the left.
\end{rem}

%

\section{Non-Discreteness Results}

\label{sec-non-discr}

\myskip\noindent
In this section we will use Shimizu's Lemma to describe those values of the angular invariant~$\al$ for which the group is not discrete,
compare also with the similar use of Shimizu's Lemma for (non ultra-parallel) complex hyperbolic $(m,n,\infty)$ groups in~\cite{Sun}, Theorem~3.7(2).
We will use the following complex hyperbolic version of Shimizu's Lemma introduced in~\cite{Par97}, Theorem~2.1:

\begin{lem}
\label{Shimizu}
Let $G$ be a discrete subgroup of $\PU(2,1)$.
Let $g\in G$ be a Heisenberg translation by~$(\xi,v)$ and $h\in G$ be an element that satisfies $h(\infty)\ne\infty$, 
then
$$r_h^2\le\rho_0(g(h^{-1}(\infty)),h^{-1}(\infty))\rho_0(g(h(\infty)),h(\infty))+4|\xi|^2,$$
where $\rho_0$ is the Cygan metric on~$\calN$ and $r_h$ is the radius of the isometric sphere of~$h$.
\end{lem}

\myskip\noindent
{\bf Proof of Proposition~\ref{non-discrete}:}
In an ultra-parallel triangle group $\<I_1,I_2,I_3\>$ we will apply Lemma~\ref{Shimizu} to the elements $g=I_2I_1$ and $h=I_3$.
Direct computation shows that 
$$
  g=I_2I_1
  =\left(\begin{matrix}
  1 & -\sqrt{2}\bar{\xi} & -|\xi|^2+iv \\ 0 & 1 & \sqrt{2}\xi \\ 0 & 0 & 1 
  \end{matrix}\right)
$$
where $\xi=2(r_1e^{-i\theta}+r_2e^{i\theta})$ and $v=8r_1r_2\sin(2\theta)$.
This is the matrix of the Heisenberg translation by~$(\xi,v)\in\calN$.
The radius of the isometric sphere of the element 
$$
h=I_3 = \left(\begin{matrix} 
0 & 0 & 1 \\ 0 & -1 & 0 \\ 1 & 0 & 0 \end{matrix}\right)
$$
is $r_h=1$.
The element~$h$ satisfies
$$
h(\infty)=h^{-1}(\infty)=[0:0:1],\quad\text{in particular}\quad h(\infty)\ne\infty.
$$
The point~$[0:0:1]\in\dd\chp$ corresponds to the point~$(0,0)\in\calN$.
The translation length of $g$ at $h(\infty)=h^{-1}(\infty)$ is 
$$
\rho_0(g(h(\infty)),h(\infty))=\rho_0(g(h^{-1}(\infty)),h^{-1}(\infty))
=\sqrt{\bigl|\,|\xi|^2-iv\bigr|}.
$$
Substituting these values in the inequality in Lemma~\ref{Shimizu} we obtain that if 
the group is discrete then
$$
1\le||\xi|^2-iv|+4|\xi|^2=\sqrt{|\xi|^4+v^2}+4|\xi|^2.
$$
Finally note that
$$|\xi|^2=|2(r_1e^{-i\theta}+r_2e^{i\theta})|^2=4(r_1^2+r_2^2+2r_1r_2\cos(2\theta)).$$
Using $\cos(2\theta)=-\cos(\al)=2\sin^2(\frac{\al}{2})-1$ and $\sin(2\theta)=\sin(\al)=2\sin(\frac{\al}{2})\cos(\frac{\al}{2})$ we obtain
\begin{align*}
  |\xi|^2&=4(r_1-r_2)^2+16r_1r_2\sin^2\left(\frac{\al}{2}\right),\quad
  v=16r_1r_2\sin\left(\frac{\al}{2}\right)\cos\left(\frac{\al}{2}\right)
\end{align*}
and hence $\sqrt{|\xi|^4+v^2}+4|\xi|^2$ is equal to 
$$4\cdot\sqrt{(r_1-r_2)^4+8r_1r_2(r_1^2+r_2^2)\sin^2\left(\frac{\al}{2}\right)}+16(r_1-r_2)^2+64r_1r_2\sin^2\left(\frac{\al}{2}\right).$$
Thus the group is not discrete if the following inequality is satisfied:
$$
\sqrt{16(r_1-r_2)^4+128r_1r_2(r_1^2+r_2^2)\sin^2\left(\frac{\al}{2}\right)}
+16(r_1-r_2)^2+64r_1r_2\sin^2\left(\frac{\al}{2}\right)<1.
$$
Rearranging and taking squares on both sides, we conclude that 
$X=64r_1r_2\sin^2(\frac{\al}{2})$
satisfies the following inequalities
\begin{equation}
\label{eq-ineq}
X^2-2bX+c>0,\quad X<d,
\end{equation}
where
\begin{align*}
b&=1-16(r_1-r_2)^2+(r_1^2+r_2^2), \quad c=\bigl(1-16(r_1-r_2)^2\bigr)^2-16(r_1-r_2)^4, \\
d&=1-16(r_1-r_2)^2.
\end{align*}
A straightforward computation shows that 
$$d^2-2bd+c=-2d(r_1^2+r_2^2)-16(r_1-r_2)^4<0.$$
If $b^2-c\ge0$ then the quadratic polynomial $X^2-2bX+c$ has real roots
and $X$ satisfies the inequalities \eqref{eq-ineq} if and only if it is less than the smaller root~$b-\sqrt{b^2-c}$ of $X^2-2bX+c$. That is:
\begin{align*}
  X<&1-16(r_1-r_2)^2+(r_1^2+r_2^2)\\
      &-\sqrt{16(r_1-r_2)^4+2\bigl(1-16(r_1-r_2)^2\bigr)(r_1^2+r_2^2)+(r_1^2+r_2^2)^2}.
\end{align*}
If $b^2-c<0$ then the polynomial $X^2-2bX+c$ has no real roots
and $X$ satisfies the inequalities \eqref{eq-ineq} if and only if $X<d$.
That is:
$$X<1-16(r_1-r_2)^2.$$
Using $X=64r_1r_2\sin^2(\alpha/2)$ and rearranging these expressions, 
we obtain the inequalities in Proposition~\ref{non-discrete}.
\qed

%


\section{Schwartz's conjecture}

\label{sec-ellipticity}

In this section we will consider our results in the context of the conjecture put forward by R.~Schwartz in his ICM talk in 2002~\cite{Sch02}:

\myskip
{\sl A complex hyperbolic \ppp-triangle group representation is discrete and faithful if and only if a group element~$w$ is non-elliptic,
where $w=\wa=I_3I_2I_3I_1$ or $w=\wb= I_1I_2I_3$ depending on~\ppp.}

\myskip
Note that Schwartz assumes $p_1\le p_2\le p_3$ which implies $r_1\le r_2\le r_3$.
We normalise differently so that $r_1\ge r_2\ge r_3=1$, hence instead of $I_3I_2I_3I_1$ the relevant element for us is $\wa=I_1I_2I_1I_3$.



\myskip
Discreteness conditions in Propositions~\ref{conditions2-typeA} and~\ref{conditions2-typeB} are 
$\trace(\wal)\ge3$ which is equivalent to~$\wal$ being not regular elliptic
and $\Re(\trace(\wb))\le -5$ which  implies that $\wb$ is non-elliptic but is not equivalent to it.
If we relax the conditions
$$\trace(\wal)\ge3\quad\text{for all}~\ell\in\z\quad\text{and}\quad\Re(\trace(\wb))\le -5$$
so that $\trace(\wal)<3$ for a single value of $\ell=k$ then the corresponding group element $\wak$ is elliptic:

\begin{prop}
\label{wak-elliptic}
Suppose that $m_1\ge m_2>0$ and for some integer $k\ge 2$
$$\max\left\{\frac{1}{k}+\frac{k+1}{r_2^2-1},\  \frac{2}{k}\right\}\le\frac{r_1^2-1}{r_2^2-1}-1\le\frac{2}{k-1},$$
where $r_j=\cosh(m_j/2)$, $j=1,2$.
If the condition
$$\max\{\fa(k+1),\fa(k-1)\}\le 4r_1r_2\sin^2\left(\frac{\al}{2}\right)<\fa(k)$$
is satisfied, then $\wak$ is elliptic and $\wal$ for $\ell\in\z\setminus\{-1,0,k\}$ are all non-elliptic.
\end{prop}

\begin{proof}
Setting $\ell_1=k+1$ in the first and $\ell_1=k-1$ in the other two cases in Lemma~\ref{lem-about-fA}
we obtain that $\fa(k+1)\ge\fa(\ell)$ for all integers $\ell>k+1$
and $\fa(\ell)\ge\fa(k-1)$ for all integers $\ell<k-1$, $\ell\ne-1,0$.
This means that our hypothesis that $4r_1r_2\sin^2(\al/2)\ge\max\{\fa(k-1),\fa(k+1)\}$ implies
$4r_1r_2\sin^2(\al/2)\ge\fa(\ell)$ for all $\ell\in\z\setminus\{-1,0,k\}$.
Recall that
$$4r_1r_2\sin^2(\al/2)\ge\fa(\ell)\iff\trace(\wal)\ge3\iff \wal~\text{is non-elliptic},$$
hence the conditions on $4r_1r_2\sin^2(\al/2)$ imply that
$\wak$ is elliptic and $\wal$ for $\ell\in\z\setminus\{-1,0,k\}$ are all non-elliptic.
\end{proof}

\myskip
With the help of this proposition we can choose $\al$ so that $\wb$ is non-elliptic, the element $\wak$ for some~$k\ge2$ is elliptic of infinite order
and all $\wal$ for $\ell\in\z\setminus\{-1,0,k\}$ are non-elliptic, in particular $\wa^{(1)}=\wa$ is non-elliptic.
Then the elements~$\wa$ and~$\wb$ are non-elliptic, but $\wak$ is elliptic of infinite order, hence the group is not discrete.
Therefore, in the ultra-parallel case, Schwartz's conjecture should be extended to include elements not only~$\wa$ and~$\wb$
but also $\wal$ with $\ell\ne1$.

\begin{example}
For $k\ge2$ let $r_1=k+1$ and $r_2=k$.
The conditions 
$$\max\left\{\frac{1}{k}+\frac{k+1}{r_2^2-1},\  \frac{2}{k}\right\}\le\frac{r_1^2-1}{r_2^2-1}-1\le\frac{2}{k-1}$$
become
$$\max\left\{(k+1)(k^2+k-1),(k+2)(k^2-1)\right\}\le k^2(k+2)\le k(k+1)^2$$
and are always satisfied.
In this case we have
$$\fb=\fa(k-1)=\fa(k+1)=0,\quad \fa(k)=\frac{1}{4k^2(k+1)^2}.$$
Proposition~\ref{conditions2-typeA} says that the group is discrete if 
$$4r_1r_2\sin^2\left(\frac{\al}{2}\right)\ge\fa(k)=\frac{1}{4k^2(k+1)^2}.$$
Proposition~\ref{wak-elliptic} implies that if 
$$4r_1r_2\sin^2\left(\frac{\al}{2}\right)<\fa(k)=\frac{1}{4k^2(k+1)^2}$$
then $\wak$ is elliptic while $\wal$ for all $\ell\ne k$ are non-elliptic.
The condition
$$4r_1r_2\sin^2\left(\frac{\al}{2}\right)\ge\fb=0$$
is always satisfied, hence $\Re(\trace(\wb)\le-5$ and the element $\wb$ is non-elliptic.
Choosing $\al$ with
$$
  \sin\left(\frac{\al}{2}\right)<\frac{1}{2k(k+1)},\quad
  \sin\left(\frac{\al}{2}\right)\ne\frac{\cos(q\cdot\pi)}{2k(k+1)}\quad\text{for all}~q\in\q,
$$
we obtain $[k+1,k,0]$-groups
with non-elliptic $\wb$ and $\wal$ for $\ell\in\z\setminus\{-1,0,k\}$,
but $\wak$ is elliptic of infinite order,
hence the group is not discrete.
\end{example}

\myskip
The condition for $\wak$ to be non-elliptic is given by an explicit inequality on the angular invariant~$\al$.
The question of where $\wb$ is non-elliptic is more subtle.
The trace
$$\trace(\wb)=-(4r_1^2+4r_2^2+1)+8r_1r_2\cdot e^{i\al}$$
is on the circle with centre~$-(4r_1^2+4r_2^2+1)$ and radius~$8r_1r_2$.
One has to carefully study the intersection of this circle and the deltoid~$\De$.

\begin{prop}
\label{wB-ell}
The element $\wb$ is non-elliptic for all values of~$\al$ if
$$7-4(r_1^2+r_2^2)+16(r_1^2-r_2^2)^2>0.$$
\end{prop}

\begin{proof}
In the case $r_1=r_2=r\ge1$ we have that
$$7-4(r_1^2+r_2^2)+16(r_1^2-r_2^2)^2=7-8r^2$$
is never positive,
hence we only need to consider the case 
$$r_1\ne r_2.$$
When does  $\trace(\wb)=8r_1r_2e^{i\alpha}-(4r_1^2+4r_2^2+1)$ lie outside the deltoid~$\De$ for all $\alpha$?
Any point where $\trace(\wb)$ lies on the deltoid is a solution to
$$8r_1r_2e^{i\alpha}-(4r_1^2+4r_2^2+1)=2e^{i\theta}+e^{-2i\theta}$$
for some $\theta$.
In other words,
\begin{align*}
  (8r_1r_2)^2
  =&\bigl| 4r_1^2+4r_2^2+1+2e^{i\theta}+e^{-2i\theta}\bigr|^2 \\
  =&\bigl(4r_1^2+4r_2^2+1+2\cos(\theta)+\cos(2\theta)\bigr)^2+\bigl(2\sin(\theta)-\sin(2\theta)\bigr)^2 \\
  =&4\bigl(2r_1^2+2r_2^2+\cos(\theta)+\cos^2(\theta)\bigr)^2+4\bigl(1-\cos(\theta)\bigr)^2\bigl(1-\cos^2(\theta)\bigr).
\end{align*}
Dividing by 4 and simplifying means that $X=\cos(\theta)\in[-1,1]$ is a root of the cubic polynomial $Q(X)$ given by
$$Q(X)=4X^3+X^2-2X+1+4(r_1^2+r_2^2)X(1+X)+4(r_1^2-r_2^2)^2.$$
Note that $Q(-1)=4(r_1^2-r_2^2)^2>0$ and $Q(1)=4+8(r_1^2+r_2^2)+4(r_1^2-r_2^2)^2)>0$.
We have
\begin{align*}
  Q'(X)
  &=12X^2+2X-2+4(r_1^2+r_2^2)(1+2X)\\
  &=2\big(3X+2(r_1^2+r_2^2)-1\big)(1+2X),
\end{align*}
hence the critical points of~$Q$ are 
$$
  X_1=\frac{1-2(r_1^2+r_2^2)}{3}<-1
  \quad\text{and}\quad
  X_2=-\frac{1}{2}.
$$
The polynomial~$Q$ satisfies $Q(-1),Q(1)>0$ and $X_2$ is the only critical point of~$Q$ in~$[-1,1]$,
hence 
$$Q(X_2)=\frac{7}{4}-(r_1^2+r_2^2)+4(r_1^2-r_2^2)^2>0$$
implies that~$Q$ has no roots in~$[-1,1]$ and hence $\trace(\wb)$ is always outside the deltoid~$\De$.
\end{proof}

\begin{rem}
Analysing the polynomial~$Q$ in the proof above, we can understand the ellipticity of~$\wb$ in other cases as well.
If $7-4(r_1^2+r_2^2)+16(r_1^2-r_2^2)^2\le0$ and $r_1\ne r_2$ then there exist $\al_1,\al_2\in(0,\pi/2)$ with $\al_1\le\al_2$ such that $\wb$ is hyperbolic for~$|\al|<\al_1$ and for $|\al|>\al_2$ and elliptic for $\al_1<|\al|<\al_2$.
If $r_1=r_2$ then there exists $\al_0\in(0,\pi/2)$ such that $\wb$ is hyperbolic for $|\al|<\al_0$ and elliptic for $|\al|>\al_0$, moreover $\al_0$ can be computed explicitly, see next section.
\end{rem}


\section{Isosceles Case}

\label{sec-isosceles}

\myskip
In this section we will give a summary of the results in the special case of an isosceles triangle, i.e.\ $m_1=m_2=m$ and $r_1=r_2=r$.

\myskip
In this case the discreteness conditions in both Propositions~\ref{conditions2-typeB} and~\ref{conditions1} become
$$\sin^2\left(\frac{\al}{2}\right)\ge\frac{1}{4r^2}=\left(\frac{1}{2r}\right)^2,$$
hence we obtain the same result as in \cite{WG}.
(To compare the results, note that $\sin^2(\al/2)=(1+t^2)^{-1}$, where $t=\tan(\theta)$ is the parameter used in~\cite{WG}.)

\myskip
On the other hand Proposition~\ref{non-discrete} says that the group is non-discrete if
$$\sin^2\left(\frac{\al}{2}\right)<\frac{2r^2+1-2r\sqrt{r^2+1}}{64r^2}=\left(\frac{r-\sqrt{r^2+1}}{8r}\right)^2.$$

\myskip
Now let us discuss the ellipticity of the elements~$\wal$ and~$\wb$.
For $r_1=r_2=r\ge1$ we have
$$4r_1r_2\sin^2\left(\frac{\al}{2}\right)\ge0\ge\frac{1-r^2}{\ell(\ell+1)}=\fa(\ell)\quad\text{and hence}\quad\trace(\wal)\ge3$$
(with equality in the case $\al=0$ and $r=1$).
Therefore $\wal$ are always non-elliptic.
As was shown in section~12 of~\cite{Pra}
the element~$\wb$ is non-elliptic for
$$\sin^2\left(\frac{\al}{2}\right)>\frac{2r^2-2}{r^2\cdot\left(64r^4-80r^2+13+(8r^2-7)^{3/2}\cdot(8r^2+1)^{1/2}\right)}.$$
(To compare the results, note that $\sin^2(\al/2)=(1+t^2)^{-1}$, where $t=(\tan(\al/2))^{-1}$ is the parameter used in~\cite{Pra}.)
Hence under the condition on $\sin^2(\al/2)$ above the elements~$\wal$  and~$\wb$ are all non-elliptic.

\begin{rem}
In the case $m_1=m_2=0$ conditions~$(1)$ and~$(2)$ imply that the ideal triangle group is discrete for $\sin(\al/2)\ge0{.}5$.
But as conjectured by Goldman and Parker~\cite{GP} and proved by Schwartz~\cite{Sch},
the ideal triangle group is still discrete 
for smaller values of~$\al$, 
namely if and only if $\sin(\al/2)\ge\frac{\sqrt{6}}{16}\approx 0{.}153$.
\end{rem}

\bigskip\noindent
{\bf Acknowledgements:}
We would like to thank the referee for their helpful comments.

\def\cprime{$'$}
\providecommand{\bysame}{\leavevmode\hbox to3em{\hrulefill}\thinspace}
\providecommand{\MR}{\relax\ifhmode\unskip\space\fi MR }
\providecommand{\MRhref}[2]{%
  \href{http://www.ams.org/mathscinet-getitem?mr=#1}{#2}
}
\providecommand{\href}[2]{#2}

\end{document}